\documentclass{amsart}

\usepackage[utf8]{inputenc}
\usepackage[english]{babel}

\usepackage{amsmath,amssymb, amsthm}
\usepackage{url}
\usepackage{nicefrac}

\newcommand{\Q}{\mathbb{Q}}
\newcommand{\Z}{\mathbb{Z}}
\newcommand{\F}{\mathbb{F}}
\DeclareMathOperator{\Gal}{Gal}
\DeclareMathOperator{\Frob}{Frob}
\DeclareMathOperator{\Nrd}{Nrd}
\DeclareMathOperator{\Trd}{Trd}
\DeclareMathOperator{\kar}{char}
\DeclareMathOperator{\trace}{Tr}
\DeclareMathOperator{\norm}{N}

\newtheorem{theorem}{Theorem}[section]
\newtheorem*{theorem*}{Theorem}

\newtheorem{proposition}[theorem]{Proposition}
\newtheorem*{proposition*}{Proposition}
\newtheorem{lemma}[theorem]{Lemma}
\newtheorem{corollary}[theorem]{Corollary}
\newtheorem*{corollary*}{Corollary}

\theoremstyle{definition}
\newtheorem{definition}[theorem]{Definition}
\theoremstyle{remark}
\newtheorem{remark}[theorem]{Remark}

\title{Irreducibility of Polynomials over Global Fields is Diophantine}
\author{Philip Dittmann}
\address{Mathematical Institute, Woodstock Road, Oxford, OX2 6GG}
\email{dittmann@maths.ox.ac.uk}

\date{19 October 2017}

\thanks{The research for this article was supported by the University of Oxford Clarendon Fund and Merton College Oxford.}

\begin{document}
\maketitle

\section{Statement of results}

We generalise methods of Poonen (\cite{poonenUniversalExistential}) and Koenigsmann (\cite{definingZinQ}) to prove the following theorem.
\begin{theorem*}
  Let $K$ be a global field, i.e.\ a number field or a function field in one variable over a finite field, and $n > 0$ a positive integer. 
  Then the set \[ \{ (a_0, \dotsc, a_{n-1}) \in K^n \colon X^n + a_{n-1}X^{n-1} + \dotsb + a_0 \in K[X] \text{ has no zero in $K$} \} \] is diophantine.
\end{theorem*}
As usual, a subset $A \subseteq K^n$ is called \emph{diophantine} if there exist $m \geq 0$ and $F \in K[X_1, \dotsc, X_n, Y_1, \dotsc, Y_m]$ such that \[ A = \{ (x_1, \dotsc, x_n) \in K^n \colon \exists y_1, \dotsc, y_m \in K (F(x_1, \dotsc, x_n, y_1, \dotsc, y_m) = 0) \}. \]
Equivalently, diophantine subsets of $K^n$ are exactly those that are definable by an existential first-order formula (with parameters) in the language of rings (equivalently, a positive existential first-order formula), and we will frequently adopt this viewpoint, in particular when seeking uniformity between different fields.

In the theorem, the construction of the polynomial $F$ -- or equivalently the defining first-order formula -- is explicit in principle, although we have not taken care to optimise the number of variables (quantifiers) necessary.

As an immediate corollary to the theorem we obtain:
\begin{corollary*}
  For every global field $K$ and $n>0$, the set of non-$n$-th powers in $K$ is diophantine.
\end{corollary*}
This was previously proven in \cite{colliotTheleneVanGeelNthPowers} in the case of a number field.
Translating the theorem into the terminology of mathematical logic yields the following:
\begin{corollary*}
  Let $K$ be a global field and $K^{\ast\ast} \supseteq K^\ast$ any two fields which are both elementary extensions of $K$. Then $K^\ast$ is relatively algebraically closed in $K^{\ast\ast}$.
\end{corollary*}
This answers Question 25 in \cite{definingZinQ}.
By a simple model-theoretic argument we also obtain:
\begin{corollary*}
  Let $K$ be a global field. There exists a diophantine criterion for a polynomial over $K$ in an arbitrary number of variables to be irreducible. More formally, fix $r,d \geq 0$. 
Then the set \begin{multline*} \{ (a_{i_1, \dotsc, i_r})_{0 \leq i_1, \dotsc, i_r \leq d} \in K^{(d+1)^r} \colon \\ \sum_{0 \leq i_1, \dotsc, i_r \leq d} a_{i_1, \dotsc, i_r} X_1^{i_1} \dotsm X_r^{i_r} \in K[X_1, \dotsc, X_r] \text{ is irreducible} \} \end{multline*} is diophantine.
\end{corollary*}

\subsection*{Acknowledgement.} I would like to thank my PhD supervisor, Jochen Koenigsmann, for helpful discussions on the subject of this paper and comments on the exposition.

\section{Preliminaries on central simple algebras}

In this section, we extend the methods pioneered in \cite{poonenUniversalExistential} from quaternion algebras over $\Q$ to general central simple algebras of prime degree over global fields.
We assume the reader to be familiar with the theory of central simple algebras; see for instance \cite{centralSimpleAlgebrasAndGalCohom} for an introduction.

Let $F$ be a field and $A$ a (finite-dimensional) central simple algebra over $F$. Define
 \[ S(A/F) = \{ \Trd(x) \colon x \in A, \Nrd(x) = 1 \}, \]
where $\operatorname{Trd}$ and $\operatorname{Nrd}$ are the reduced trace and norm, respectively.

\begin{proposition}\label{propositionLocalGlobalPrinciple}
  Let $L$ be a global field and $A$ a central simple algebra over $L$ of prime degree $l$. Then
		\[ S(A/L) = \bigcap_{v \text{ a place of } L} S(A \otimes L_v / L_v) \cap L. \]
\end{proposition}
This proposition is well-known for $\kar L \neq 2$ and $l = 2$ (the proof in this case can rely on the Hasse principle for quadratic forms), and this case has been exploited for first-order definitions of $\Z$ in $\Q$ (and more generally for rings of integers in number fields, see \cite{parkDefiningIntegers}, since adapted to global fields of odd characteristic, see \cite{eisentraegerMorrison}). 
For the proof in the general case we quote two lemmas from the theory of central simple algebras:
\begin{lemma}[{\cite[Theorem 4.12]{jacobsonBasicAlgebraII}}]\label{lemmaSplittingFieldIffSubfield}
	Let $D$ be a central division algebra of degree $n$ over a field $F$, and let $F'$ be a
field of degree $n$ over $F$. Then $F'$ splits $D$ if and only if $F'$ can be embedded into $D$ over $F$ (i.e.\ there is a subalgebra of $D$ isomorphic to $F'$ over $F$).
\end{lemma}
\begin{lemma}[{\cite[Proposition 2.6.3]{centralSimpleAlgebrasAndGalCohom}}]\label{lemmaComputingNrdAndTrd}
	Let $A/F$ be a central simple algebra of degree $n$ and $x \in A$. If $F' \subseteq A$ is commutative subalgebra which contains $x$ and is a degree $n$ field extension of $F$, then $\Nrd(x) = \norm_{F'/F}(x)$ and $\Trd(x) = \trace_{F'/F}(x)$. 
	In particular, if $n$ is prime and $A$ is a division algebra, then $\Nrd(x) = x^n$ and $\Trd(x) = nx$ if $x \in F$ and otherwise $\Nrd(x) = \norm_{F(x)/F}(x)$ and $\Trd(x) = \trace_{F(x)/F}(x)$.
\end{lemma}
In this lemma, as in the sequel, we write $F(x)$ for the smallest subalgebra of $A$ containing $F$ and $x$; under the assumption that $A$ is a division algebra of prime degree $n$, $F(x)$ is necessarily a commutative division algebra, i.e.\ a field. 
Then the degree $[F(x) \colon F]$ is either $1$ or $n$, since $A$ is a (left) $F(x)$-vector space and $n^2 = \dim_F A = \dim_{F(x)} A \cdot [F(x) \colon F]$ (cf.\ the usual tower law for field extensions).

Lastly, we need the following easy consequence of Krasner's Lemma:
\begin{lemma}\label{lemmaGlobalPolynomialGivenLocals}
  Let $l$ be a prime number, $L$ a global field, $a \in L$, and $v_1, \dotsc, v_r$ places of $L$. For each $i$, let $f_i \in L_{v_i}[X]$ be a monic irreducible polynomial of degree $l$ with constant coefficient $(-1)^l$ and $X^{l-1}$-coefficient $-a$. Then there exists a monic irreducible polynomial $f \in L[X]$ of degree $l$ with constant coefficient $(-1)^l$ and $X^{l-1}$-coefficient $-a$ such that $L_{v_i}[X]/(f) \cong L_{v_i}[X]/(f_i)$, i.e.\ the completion above $v_i$ of the global field $L[X]/(f)$ is unique and given by $L_{v_i}[X]/(f_i)$.
\end{lemma}
\begin{proof}
  Each $f_i$ must automatically be separable -- otherwise we would have to have $l = \kar L$, $f_i = X^l + (-1)^l$, but then $f_i$ would be reducible.

  If the coefficients of $f$ are $v_1$-adically sufficiently close to those of $f_1$, then $f$ is irreducible in $L_{v_1}[X]$ and therefore in $L[X]$, and additionally $L_{v_1}[X]/(f) \cong L_{v_1}[X]/(f_1)$: This is a standard consequence of Krasner's Lemma, see e.g.\ \cite[Exercise 9.8.7]{jacobsonBasicAlgebraII}; see also \cite[Lemma 8.1.6 and proof of Proposition 8.1.5]{NeukirchSchmidtWingberg}. (For archimedean places $v_i$, the statement is easily checked separately.)
 
By weak approximation, we can choose the coefficients of $f$ to $v_i$-adically approximate the coefficients of $f_i$ arbitrarily well simultaneously for all $i$.
\end{proof}

\begin{proof}[Proof of Proposition \ref{propositionLocalGlobalPrinciple}]
	This is a local-global principle. Note that there is nothing to show if $A$ is isomorphic to the algebra of $l \times l$-matrices over $L$ since the set on both sides is just $L$ in this case. So let us assume that $A$ is non-split over $L$; hence, since $A$ has prime degree over $L$, $A$ is a division algebra by Wedderburn's Theorem (see e.g.\ \cite[Theorem 2.1.3]{centralSimpleAlgebrasAndGalCohom}).
        The inclusion $\subseteq$ is clear. For the other inclusion, consider an element $a \in L$ of the right-hand side. We want to show that $a \in S(A/L)$, i.e.\ that there exists $x \in A$ with $\operatorname{Nrd}(x) = 1$ and $\operatorname{Trd}(x) = a$. Let $v_1, \dotsc, v_r$ be the ramified places of $A$. (The local condition is trivial at all other places.) For each $v_i$ there exists an element $x_i \in A \otimes L_{v_i}$ of reduced norm $1$ and reduced trace $a$. 
We can disregard the case where $x_i$ is in the centre $L_{v_i}$ -- the norm condition then forces $x_i^l = 1$, and if $l \neq \kar L$ then the trace condition forces $x_i = \frac a l \in L$, or if $l = \kar L$ then $x_i = 1 \in L$; in either case we are done globally.
        
Assume therefore that $x_i \not\in L_{v_i}$, so $L_{v_i}(x_i) / L_{v_i}$ is a field extension of degree $l$ which splits $A \otimes L_{v_i}$ (by Lemma \ref{lemmaSplittingFieldIffSubfield}), with $\norm_{L_{v_i}(x)/L_{v_i}}(x) = 1$ and $\trace_{L_{v_i}(x)/L_{v_i}}(x) = a$ (by Lemma \ref{lemmaComputingNrdAndTrd}). Write $f_i \in L_{v_i}[X]$ for the minimal polynomial of $x_i$; it is a monic irreducible polynomial of degree $l$ with constant coefficient $(-1)^l$ and $X^{l-1}$-coefficient $-a$ for reasons of norm and trace.

Let $f \in L[X]$ be a polynomial as in Lemma \ref{lemmaGlobalPolynomialGivenLocals}, so $L' = L[X]/(f)$ is a degree $l$ field extension of $L$ with completions $L_{v_i}[X]/(f_i) \cong L_{v_i}(x_i)$ above each $v_i$. The element $\overline{X} \in L'$ has minimal polynomial $f$, hence $\norm_{L'/L}(\overline{X}) = 1$ and $\trace_{L'/L}(\overline{X}) = a$. The field $L'$ splits $A$ by the Hasse-Brauer-Noether Theorem since it splits $A$ everywhere locally, so $L'$ embeds into $A$ by Lemma \ref{lemmaSplittingFieldIffSubfield}, and the image of $\overline{X}$ under this embedding has reduced norm $1$ and reduced trace $a$ by Lemma \ref{lemmaComputingNrdAndTrd}.
\end{proof}
This proof idea is already present in \cite[Theorem 3.1]{eisentraegerIntegralityAtAPrime}.

Next we investigate the local conditions $S(A/L_v)$. 
For a finite field $\F$, define \[ U_l(\F) = \{ \trace_{\F^{(l)}/ \F}(x) \colon x \in \F^{(l)} \setminus \F, \norm_{\F^{(l)}/ \F}(x) = 1 \} \subseteq \F,\] where we write $\F^{(l)}$ for the extension field of $\F$ of degree $l$ (unique up to isomorphism).

The following is essentially \cite[Lemma 2.1]{poonenUniversalExistential}.
\begin{lemma}\label{lemmaFirstLocalObservations}
  Let $A/F$ be a central simple algebra of prime degree $l$ over a local field.
  \begin{enumerate}
    \item If $A$ is split, then $S(A/F) = F$.
    \item If $A$ is a division algebra, then for all irreducible monic polynomials $X^l + a_{l-1}X^{l-1} + \dotsb + a_0$ with constant coefficient $a_0 = (-1)^l$ we have $-a_{l-1} \in S(A/F)$.
    \item If $A$ is a division algebra and $F$ is non-archimedean, then $\mathcal{O} \supseteq S(A/F) \supseteq \operatorname{res}^{-1}(U_l(\F))$, where $\mathcal{O}$ is the valuation ring, $\F$ is the residue field, and $\operatorname{res} \colon \mathcal{O} \to \F$ is the residue map.
  \end{enumerate}
\end{lemma}
\begin{proof}
  In the split case reduced norm and trace coincide with the usual matrix determinant and trace, and all monic polynomials of degree $l$ do occur as characteristic polynomials.

For the second point, every monic irreducible polynomial $f$ generates a field extension $F[x]/(f)$, and every such field extension splits $A$ by the theory of central simple algebras over local fields (see e.g.\ \cite[Corollary 7.1.4]{NeukirchSchmidtWingberg}). Hence, by Lemma \ref{lemmaSplittingFieldIffSubfield}, $F[x]/(f)$ embeds into $A$, and then $-a_{l-1} \in S(A/F)$ by Lemma \ref{lemmaComputingNrdAndTrd}, where $a_{l-1}$ is the coefficient of $X^{l-1}$ in $f$.

  For the case of a division algebra over a non-archimedean local field, let $x \in A$ with $\Nrd(x) = 1$. Then by Lemma \ref{lemmaComputingNrdAndTrd} $\norm_{F(x)/F}(x)^l = 1$, so $x$ is contained in the valuation ring of the local field $F(x)$, therefore integral over $\mathcal{O}$ and hence has integral trace. This proves $S(A/F) \subseteq \mathcal{O}$. 
For the other inclusion, if $\overline{a} \in U_l(\F)$ then there exists a monic irreducible polynomial $\overline{f} = X^n + \overline{a_{l-1}}X^{l-1} + \dotsb + \overline{a_0} \in \F[X]$ with $\overline{a_{l-1}} = - \overline{a}$ and $\overline{a_0} = (-1)^l$. Any lift of $\overline f$ to $F[X]$ is irreducible over $F$ (because it generates an unramified extension of degree $l$), hence any lift $a \in \mathcal{O}$ of $\overline{a}$ is in $S(A/F)$ by the second point.
\end{proof}

We can now give a satisfactory statement on $S(A/F)$ in the non-split local case.
\begin{proposition}\label{propositionLocalSComputation}
  Let $A/F$ be a central division algebra of prime degree $l$ over a non-archimedean local field.
  \begin{enumerate}
    \item If $l > 2$, then $S(A/F)$ is equal to the valuation ring $\mathcal{O}$ of $F$.
    \item If $l = 2$, write $V(A/F)$ for the topological interior of $S(A/F)$. We then have $V(A/F) - V(A/F) = \mathcal{O}$, where $V(A/F) - V(A/F)$ is the set of differences of two elements of $V(A/F)$.
  \end{enumerate}
\end{proposition}

For the proof of the first point, it suffices to prove the following lemma:
\begin{lemma}\label{lemmaUComputation}
  For $l > 2$ and an arbitrary finite field $\F$, we have $U_l(\F) = \F$.
\end{lemma}
\begin{proof}
  This result is equivalent to showing that for any given $a \in \F$, there exists a monic irreducible polynomial $f \in \F[X]$ of degree $l$ with $X^{l-1}$-coefficient $-a$ and constant coefficient $-1$. 
  Let us write $q$ for the cardinality of $\F$. If $l > 5$, or $l = 5$ and $q > 9$, the result follows from Theorem \ref{thmExistenceOfIrreduciblePolynomials} below. The remaining cases for $l = 5$ we check by hand.

It remains to consider the case $l = 3$. If a polynomial $f_b = X^3 - aX^2 + bX - 1$ is not irreducible, it must be divisible by $X-c$ for some $c \in \F^\times$. However, for each $c$ there exists exactly one $f_b$ divisible by $X-c$. By counting, there exists an $f_b$ which is not divisble by any $X-c$ and hence irreducible.
\end{proof}

\begin{theorem}\label{thmExistenceOfIrreduciblePolynomials}
  Let $\F$ be a finite field of cardinality $q$ and $n > 0$.
  \begin{enumerate}
    \item If $n \geq 5$ and $q > \big(\frac{n+1}{2}\big)^2$, there exists a monic irreducible polynomial of degree $n$ over $\F$ with any given non-zero constant coefficient and given $X^{n-1}$-coefficient.
    \item If $n \geq 6$, the same is true without assumption on $q$.
  \end{enumerate}
\end{theorem}
\begin{proof}
  These are Corollaries 2.2 and 2.3 of \cite{explicitTheoremsOnGeneratorPolynomials}.
\end{proof}

\begin{proof}[Proof of Proposition \ref{propositionLocalSComputation}]
   For $l > 2$, the claim follows from the third part of Lemma \ref{lemmaFirstLocalObservations} and Lemma \ref{lemmaUComputation}, so let us consider $l=2$. 
Write $\F$ for the residue field of $F$ and pick a uniformiser $\pi$.

For $a \in 2 + \pi + \pi^2 \mathcal{O}$, the polynomial $f = X^2 - aX + 1$ is irreducible, since $f(X+1) = X^2 - (a-2)X + (2-a)$ is irreducible by Eisenstein's criterion. 
Hence, by the second point of Lemma \ref{lemmaFirstLocalObservations}, we have $2 + \pi + \pi^2 \mathcal{O} \subseteq S(A/F)$, so $V(A/F)$ contains the element $b = 2 + \pi$; likewise $-b \in V(A/F)$ by passing from $f$ to $f(-X)$.
Furthermore, we also have $V(A/F) \supseteq \operatorname{res}^{-1}(U_2(\F))$ by Lemma \ref{lemmaFirstLocalObservations} since the right-hand side is open. 

If $\F$ has cardinality greater than $11$, then $U_2(\F) - U_2(\F) = \F$ by \cite[Lemma 2.3]{poonenUniversalExistential}. 
We can check exhaustively that for the remaining finite fields we have $(U_2(\F) \cup \{ 2, -2 \}) - U_2(\F) = \F$. 
Hence \[ V(A/F) - V(A/F) \supseteq \big(\operatorname{res}^{-1}(U_2(\F)) \cup \{ b, -b \}\big) - \operatorname{res}^{-1}(U_2(\F)) = \mathcal{O}. \qedhere \]
\end{proof}
This proof is adapted from the proof of Proposition 2.3 in \cite{parkDefiningIntegers}. 
A modification is necessary because the set $V_v$ constructed there fails to be contained in the interior of $S(A/F)$, interfering with the application of approximation theorems later on.

For a central simple algebra $A$ of prime degree $l$ over a global field $L$, we define $T(A/L) = S(A/L)$ if $l > 2$ and $T(A/L) = S(A/L) - S(A/L)$ (the set of pairwise differences of elements of $S$) if $l = 2$.
\begin{proposition}\label{propositionTIsSemilocal}
   If $A/L$ splits at all real places of $L$ (which is always the case if $l \neq 2$ or $L$ is a global function field), then \[ T(A/L) = \bigcap_{\mathfrak q \in \Delta_{A/L}} \mathcal{O}_{\mathfrak q} \cap L, \] where $\Delta_{A/L}$ is the finite set of places of $L$ at which $A/L$ does not split.
\end{proposition}
\begin{proof}
  For $l > 2$, this is immediate from Propositions \ref{propositionLocalGlobalPrinciple} and \ref{propositionLocalSComputation}, so consider the case $l = 2$.
  The inclusion $\subseteq$ is clear from Proposition \ref{propositionLocalGlobalPrinciple} and Lemma \ref{lemmaFirstLocalObservations}, so let $x \in \bigcap_{\mathfrak q \in \Delta_{A/L}} \mathcal{O}_{\mathfrak q} \cap L$. For each $\mathfrak q \in \Delta_{A/L}$ we have $x \in \mathcal{O}_{\mathfrak q} = V(A/L_{\mathfrak q}) - V(A/L_{\mathfrak q})$ according to Proposition \ref{propositionLocalSComputation}, so pick $a_{\mathfrak q} \in V(A/L_{\mathfrak q})$ such that $x + a_{\mathfrak q} \in V(A/L_{\mathfrak q})$. Since $\Delta_{A/L}$ is finite and the $V(A/L_{\mathfrak q})$ are open, we can use weak approximation to find $a \in L$ such that $a, x+a \in V(A/L_{\mathfrak q}) \subseteq S(A/L_{\mathfrak q})$ for all $\mathfrak q \in \Delta(A/L)$, hence $a, x+a \in S(A/L)$ by Proposition \ref{propositionLocalGlobalPrinciple} and therefore $x \in T(A/L)$.
\end{proof}

For later use, we also record the following fact:
\begin{proposition}\label{propInseparableExtensionSplitsAllAlgebras}
  Assume $K$ is global field of characteristic $p > 0$ and $L/K$ is a finite inseparable extension. Then any central simple algebra $A/K$ of degree $p$ is split by $L$.
\end{proposition}
\begin{proof}
  By replacing $K$ with the maximal separable subextension of $L/K$, we may assume that $L/K$ is a purely inseparable proper extension.
  Since $[K^{\nicefrac 1p} \colon K] = p$, we now necessarily have $L \supset K^{\nicefrac 1p}$. Hence the result follows from Lemma \ref{lemmaInseparableSplitting} below.
\end{proof}

\begin{lemma}[{\cite[Theorem 4.1.8]{jacobsonFinDimDivAlgs}}] \label{lemmaInseparableSplitting}
  Let $F$ be a field of characteristic $p > 0$ and $A/F$ be a central simple algebra of degree $p$. Then $A$ is split by the field $F^{\nicefrac 1p}$.
\end{lemma}

\subsection{First-order definability}

For what is to follow we will need that $S(A/F)$ is existentially first-order definable (in the language of rings) in terms of structure constants of $A$, i.e.\ $S(A/F)$ needs to be diophantine uniformly over central simple algebras of some prime degree $l$ and uniformly over base fields.
(Recall that structure constants with respect to a basis $(X_i)_{1 \leq i \leq l^2}$ of $A/F$ are the constants $(a_{ijk})_{1 \leq i,j,k \leq l^2}$ in $F$ such that $X_i \cdot X_j = \sum_k a_{ijk} X_k$.) 
It suffices to prove the following lemma.
\begin{lemma}
  The functions $\Trd$ and $\Nrd$ are uniformly quantifier-freely definable, i.e.\ for fixed $l$ there exists a quantifier-free first-order formula $\tau$ in $l^6 + l^2 + 1$ variables such that if $F$ is a field, $(a_{ijk})_{1 \leq i,j,k \leq l^2}$ are structure constants of a central simple algebra $A/F$, $(b_i)_{1 \leq i \leq l^2}$ is the basis expansion of an element $x \in A$, and $c \in F$ is arbitrary, then $F \models \tau(\overline a, \overline b, c)$ if and only if $c = \Trd(x)$, and there is likewise such a formula for $\Nrd$ in place of $\Trd$.
\end{lemma}
\begin{proof}
  If $(a_{ijk})$ are the structure constants of $M_{l \times l}$ in the standard basis (i.e.\ the basis given by matrices with a single entry equal to $1$ and all other entries $0$), then the elements $(b_i)$ are precisely the entries of the matrix $x \in M_{l \times l}$, and hence reduced norm and trace, which agree with matrix determinant and trace, are given by polynomial functions of the $b_i$. 

  If $F$ is algebraically closed, then any central simple algebra of degree $l$ is isomorphic to $M_{l \times l}$, i.e.\ there exists a base change matrix that transforms the situation to the previous one. Hence we can find an existential formula $\tau$ that works over algebraically closed fields -- it asserts the existence of a base change transforming the $a_{ijk}$ into the structure constants with respect to the standard basis, and $c$ being the right polynomial function of the (transformed) $b_i$. By quantifier elimination, we can replace $\tau$ by a quantifier-free formula which is equivalent over algebraically closed fields.

  This formula $\tau$ in fact works for all fields. For this it suffices to note that $\Trd_{A/F}(x) = \Trd_{A \otimes \overline{F} / \overline{F}}(x \otimes 1)$ and likewise for the reduced norm, and $F \models \tau(\overline a, \overline b, c)$ if and only if $\overline{F} \models \tau(\overline a, \overline b, c)$ since $\tau$ is quantifier-free.
\end{proof}

Consequently, the set $T(A/L)$ is existentially first-order definable, uniformly over $A$ of some fixed prime degree $l$ and global fields $L$.
We can even require the defining formula to be positive and existential, since an inequality $x \neq 0$ may always be replaced by the equivalent $\exists y (x \cdot y = 1)$.

\begin{corollary}\label{corollaryOpDefinable}
  For any finite place $\mathfrak p$ of a global field $K$, the ring $\mathcal{O}_{\mathfrak p} \cap K$ is diophantine in $K$.
\end{corollary}
\begin{proof}
  Take a prime number $l$, e.g.\ $l=2$, and pick two central simple algebras $A,A'/K$ of degree $l$ splitting at all real places of $K$ and such that $\Delta_{A/K} \cap \Delta_{A'/K} = \{ \mathfrak p \}$. This is always possible by the characterisation of the Brauer group of global fields by Hasse invariants, see \cite[Theorem 8.1.17]{NeukirchSchmidtWingberg}.
Then $T(A/K) + T(A'/K) = \mathcal{O}_{\mathfrak p} \cap K$, and this is positively existentially definable in $K$ with parameters.
\end{proof}
Of course, this result is far from new -- see for instance \cite[Theorem 4.4]{diophantineClassesOfHolomorphyRingsOfGlobalFields}. The proof is essentially the same as \cite[Remark 2.6]{poonenUniversalExistential}.

For later use, we also spell out a definability result for cyclic algebras. Recall that for a field $F$ and a cyclic extension $M/F$ of degree $l$ with a generator $\sigma$ of $\Gal(M/F)$ and an element $a \in F$ the cyclic algebra $(M,\sigma,a)$ is the $F$-algebra generated by $M$ and an element $y$ subject to the relations $y^l = a$ and $xy = y \sigma(x)$ for all $x \in M$. It is central simple of degree $l$.

\begin{lemma}\label{lemmaCyclicTDefinable}
  Let $K$ be a global field and $M/K$ a cyclic extension of prime degree $l$ with a generator $\sigma$ of $\Gal(M/K)$. Then there is a positive existential formula $\varphi_{M,\sigma}(a, x)$ in the language of rings, with parameters from $K$, such that in any finite extension $L/K$ and for any $a \in K$ the formula $\varphi_{M,\sigma}(a, \cdot)$ defines the set $T((M,\sigma,a) \otimes_K L / L)$.
\end{lemma}
\begin{proof}
  Fixing an irreducible polynomial over $K$ with splitting field $M$, we can write structure constants for $(M, \sigma, a)$ as polynomial expressions in $a$ with parameters from $K$. These are then also structure constants for $(M,\sigma,a) \otimes_K L / L$. Now we can use the positive existential definition of $T$.
\end{proof}

\section{An Interlude in Class Field Theory}

Fix a global field $K$ and $n > 1$. Also fix a prime $l \mid n$ for this section.

The entirety of this section is rather technical; we set up the necessary machinery from class field theory -- notably describing certain ideal groups $I_{\mathfrak m}$ and $H$, as well as field extensions $M_i/K$ -- which is needed for our main proofs in the next section, in particular the central Proposition \ref{propositionCriterionForProperExtension}.

Let us fix some notation. Write $\Sigma$ for the set of places of $K$. If $K$ is a number field, write $\Sigma_\infty \subset \Sigma$ for the set of archimedean places. If $K$ is a global function field, arbitrarily fix $\Sigma_\infty$ to be any finite non-empty subset of $\Sigma$. In either case, we call $\Sigma_\infty$ the set of places at infinity.

Let $\mathcal{O}_K$ be the ring of elements of $K$ integral at each place in $\Sigma \setminus \Sigma_\infty$; this is a Dedekind domain, and the prime ideals of $\mathcal{O}_K$ are in bijection to places in $\Sigma \setminus \Sigma_\infty$. 
This ring is the usual ring of integers in the number field case. In the case of function fields, $\mathcal{O}_K$ depends on the choice of $\Sigma_\infty$.

Write $I_{\mathcal{O}_K}$ for the group of fractional ideals of $\mathcal{O}_K$, $P_{\mathcal{O}_K}$ for the subgroup of principal fractional ideals, and $\operatorname{Cl}(\mathcal{O}_K) = I_{\mathcal{O}_K}/P_{\mathcal{O}_K}$. 
In the number field case, this is the usual ideal class group and well-known to be finite. 
In the function field case, this is not the usual divisor class group, since we are ignoring the places at infinity, but rather the $\Sigma_\infty$-class group in the sense of \cite[chapter 14]{rosenFunctionFields} -- essentially the divisor class group modulo the classes of prime divisors at infinity. It is finite by Corollary 2 to Proposition 14.1 ibid.

We now fix some field extensions of $K$ for later use.
Choose $k$ such that $l^k > \lvert\operatorname{Cl}(\mathcal{O}_K)\rvert \cdot n!$. Find an abelian extension $M/K$ with Galois group $\Gal(M/K) \cong (\Z/l\Z)^k$ and such that $M/K$ is completely split at all infinite places.
\begin{lemma}
  For any choice of $k$ we can find such $M$.
\end{lemma}
\begin{proof}
  This follows from existence theorems in class field theory, e.g.\ the general version of the Grunwald-Wang Theorem \cite[Theorem 9.2.8]{NeukirchSchmidtWingberg}. (Note that we are never in what is called the ``special case'' there, since we are looking for an abelian extension whose Galois group has prime exponent.)

It is not hard to give an explicit argument in the present situation, using (the totally real part of) cyclotomic extensions in the number field case, and the analogous Carlitz module construction (see \cite[chapter 12]{rosenFunctionFields}) over a suitable subfield $\F_p(T) \subseteq K$ in the function field case.
\end{proof}

\begin{remark}
  This choice of a distinguished abelian extension of $K$ is already present in previous papers, in the special case $l = k = 2$; most notably in subsection 3.3 of \cite{parkDefiningIntegers}, a field extension $K(\sqrt a, \sqrt b)/K$ is chosen. Likewise, the modulus $8$ which appears throughout \cite{definingZinQ} can be retrospectively explained by an implicit choice of field extension $\Q(\sqrt 2, \sqrt{-1})/\Q$.
  The paper \cite{eisentraegerMorrison} independently from us transfers some of the ideas of \cite{parkDefiningIntegers} to the setting of global function fields. Note however that in our situation the analogy between function fields and number fields is more direct: We do not have to impose the condition that $M$ be linearly disjoint from the Hilbert class field of $K$ as in (the proof of) \cite[Lemma 3.19]{parkDefiningIntegers} -- a condition that \cite{eisentraegerMorrison} changes in the function field situation.
\end{remark}

Let us write $I_{\mathfrak m} \leq I_{\mathcal{O}_K}$ for the set of fractional ideals of $\mathcal{O}_K$ in which none of the prime ideals ramified in $M/K$ occur in numerator or denominator. 
Then we obtain the well-known \emph{Artin map} \[ I_{\mathfrak m} \to \Gal(M/K) \] as the unique homomorphism sending an unramified prime ideal to its Frobenius element. (In the function field case, note that since all infinite places are completely split in $M/K$, this map is induced by the Artin map on divisors.)
Write $H < I_{\mathfrak m}$ for the kernel of this map.
 
By the Chebotarev Density Theorem, the set of prime ideals (excluding those at infinity and ramified ones) mapping to a given element of $\Gal(M/K)$ -- put otherwise, in a given coset in $I_{\mathfrak m}/H$ -- has density $\frac{1}{\lvert\Gal(M/K)\rvert} = \frac{1}{\lvert I_{\mathfrak m} / H\rvert} = l^{-k}$. (Throughout, it does not matter whether we choose natural or Dirichlet density.)

By class field theory -- see e.g.\ \cite[X, §2]{langANT} for number fields and \cite[Theorem 9.23]{rosenFunctionFields} for function fields -- there exists a modulus or cycle $\mathfrak m = \sum_{\mathfrak p} n_{\mathfrak p} \mathfrak{p}$, a formal sum of places of $K$ ramified in $M$ with $n_{\mathfrak p} \geq 0$, such that $H$ contains the subgroup $P_{\mathfrak m} = \{ (a) \colon a \in U_{\mathfrak m} \}$, where
\[ U_{\mathfrak m} = \{ a \in K^\times \colon v_{\mathfrak p}(a - 1) \geq n_{\mathfrak p} \text{ for all ramified $\mathfrak p$} \} .\]
The quotient group $I_{\mathfrak m} / P_{\mathfrak m}$, a \emph{generalised ideal class group}, is finite.

Now choose subextensions $M_1, \dotsc, M_k$ with $\Gal(M_i/K) \cong \Z/l\Z$ such that $M$ is the composite of the $M_i$. Furthermore, fix a generator $\sigma_i$ of $\Gal(M_i/K)$ for each $i$.

The rest of this section consists of two lemmas needed in the proof of Proposition \ref{propositionCriterionForProperExtension}.

\begin{lemma}\label{lemmaNonsplitInK}
  Let $a \in K^\times$ such that $(a) \in I_{\mathfrak m}$ and $(a) \not\in H$. Then there exist a place $\mathfrak{p} \not\in \Sigma_\infty$ and an $M_i$ such that the algebra $(M_i, \sigma_i, a)$ is not split at $\mathfrak{p}$, and $a \not\in \mathcal{O}_{\mathfrak p}^\times$.
\end{lemma}
\begin{proof}
  The fractional ideal $(a)$ of $\mathcal{O}_K$ factors as a product of prime ideals of $K$ unramified in $M$ and not in $\Sigma_\infty$. 
  The group $I_{\mathfrak m} / H \cong \Gal(M/K) \cong (\Z/l\Z)^k$ has exponent $l$, hence there exists a prime ideal $\mathfrak p \not\in H$ that occurs in $(a)$ with multiplicity not divisible by $l$ since $(a) \not\in H$.

  The prime $\mathfrak p$ is not completely split in $M$ since $\mathfrak p \not\in H$, so there exists some $M_i$ in which $\mathfrak p$ is inert, i.e.\ the local extension $M_i K_{\mathfrak p} / K_{\mathfrak p}$ is unramified of degree $l$.
  Therefore the group of local norms $\norm_{M_i K_{\mathfrak p} / K_{\mathfrak p}}((M_i K_{\mathfrak p})^\times) \subseteq K_{\mathfrak p}^\times$ consists of the elements of $l$-divisible valuation; thus $a$ is not a local norm and therefore $(M_i, \sigma_i, a)$ is not split at $\mathfrak p$.
\end{proof}

\begin{lemma}\label{lemmaSplitInExtension}
  Let $P \subset \Sigma \setminus \Sigma_\infty$ be a set of places of density at least $\frac{1}{n!}$. Then there exists $a \in K^\times$ such that $(a) \in I_{\mathfrak m}$, $(a) \not\in H$ and all places $\mathfrak p \in \Sigma \setminus \Sigma_\infty$ with $a \not\in \mathcal{O}_{\mathfrak p}^\times$ are in $P$.
\end{lemma}
\begin{proof}
  We may remove the finitely many places ramified in $M/K$ from $P$ without affecting the hypotheses.

  The set $P$ has density at least $\frac{1}{n!} > \lvert \operatorname{Cl}(\mathcal{O}_K) \rvert / \lvert I_{\mathfrak m} / H \rvert$. 
  Since the set of prime ideals in each coset in $I_{\mathfrak m}/H$ has density $1/ \lvert I_{\mathfrak m} / H \rvert$ as noted above, $P$ contains prime ideals from at least $\lvert \operatorname{Cl}(\mathcal{O}_K) \rvert + 1$ different cosets; thus we may pick $\mathfrak p, \mathfrak{p}' \in P$ in different classes in $I_{\mathfrak m} / H$ and in the same class in $\operatorname{Cl}(\mathcal{O}_K)$.

  Now $\mathfrak p {\mathfrak{p}'}^{-1}$ is a principal fractional ideal of $\mathcal{O}_K$; pick a generator $a$. By construction, this generator satisfies all of the requirements.
\end{proof}

\section{A diophantine criterion for proper extensions of global fields}

In this section we find an existential sentence that distinguishes the fixed global field $K$ from its finite extensions of degree $n$, see Theorem \ref{theoremExistentialFormulaForProperExtension}.

\begin{definition}
  Let $L/K$ be an extension of degree $n$ and $l \mid n$ a prime number. A prime ideal $\mathfrak p$ of $K$ is \emph{$l$-good} (for $L$) if it is unramified in $L$ and for all prime ideals $\mathfrak q$ of $L$ above $\mathfrak p$ the inertia degree $[\mathcal{O}_L/\mathfrak q : \mathcal{O}_K/\mathfrak p]$ is divisible by $l$.

  The prime number $l \mid n$ is \emph{admissible} (for $L$) if either
  \begin{itemize}
    \item $L/K$ is separable and the set of $l$-good prime ideals of $K$ has density at least $\frac 1 {n!}$, or
    \item $L/K$ is inseparable and $l = \kar K$.
  \end{itemize}
\end{definition}

\begin{lemma}\label{lemmaGoodness}
  For every $L/K$ of degree $n$ there exists an admissible $l \mid n$.
\end{lemma}
\begin{proof}
  If $L/K$ is inseparable, then by basic field theory $\kar K \mid n$, so $l = \kar K$ is admissible. 
Let us now assume that $L/K$ is separable.

  Let $L'/K$ be the Galois hull of $L/K$, $G = \Gal(L'/K)$, $H = \Gal(L'/L) \lneq G$. Then $\lvert G \rvert \leq n!$. Let $g \in G$ of prime power order $l^r$ with $l \mid n$ such that no conjugate of $g$ is in $H$. 
Existence of such $g$ is assured by Theorem \ref{theoremFeinKantorSchacher} below: An element $g$ has no conjugate in $H$ if and only if $g$ has no fixed point in the left-multiplication action of $G$ on $\Omega = G/H$.
If $\mathfrak q$ is a prime ideal of $L'$ above an unramified ideal $\mathfrak p$ of $K$ such that $\Frob(\mathfrak q/\mathfrak p)$ is conjugate to $g$, then the inertia degree $f(\mathfrak q / \mathfrak p)$ is equal to $\operatorname{ord}(g) = l^r$, so for $\mathfrak q' = \mathfrak q \cap L$ we have $f(\mathfrak q' / \mathfrak p) \neq 1$ since $\Frob(\mathfrak q/\mathfrak p) \not\in H$, and $f(\mathfrak q' / \mathfrak p) \mid l^r$, hence $l \mid f(\mathfrak q'/\mathfrak p)$. 
The set of such prime ideals $\mathfrak p$ has density at least $\frac 1 {n!}$ by the Chebotarev Density Theorem.
\end{proof}

\begin{theorem}[Fein-Kantor-Schacher]\label{theoremFeinKantorSchacher}
  Let $G$ be a finite group acting transitively on a set $\Omega$ with $\lvert\Omega\rvert > 1$. Then there exists an element $g \in G$ of prime power order $l^r$, with $l \mid \lvert\Omega\rvert$, acting without fixed points on $\Omega$.
\end{theorem}
\begin{remark} The paper \cite{feinKantorSchacher}, in which this theorem was first proved, used it for a similar purpose as we do -- classifying relative Brauer groups $\operatorname{Br}(L/K)$ of global fields. 
There appears to be no known proof of this theorem that does not use the classification of finite simple groups. 
\end{remark}

\begin{proposition}\label{propositionCriterionForProperExtension}
  For a global field $L/K$ consider the following statement, which we call $(\dagger)^l_{L/K}$:
  \begin{quotation}
    There exists an element $a \in K^\times$ such that $(a) \in I_{\mathfrak m}$, $(a) \not\in H$ and for all $i$ both $a$ and $\frac{1}{a}$ are in $T((M_i, \sigma_i, a) \otimes_K L / L)$.
  \end{quotation}
  Then this statement is false for $L = K$, and it is true if $L/K$ is an extension of degree $n$ with $l$ admissible.
\end{proposition}
\begin{proof}
  Let us first consider the case $L = K$, and assume there were $a$ as in the statement. By Lemma \ref{lemmaNonsplitInK} there exist an $M_i$ and a place $\mathfrak p \not\in \Sigma_\infty$ such that the algebra $(M_i, \sigma_i, a)$ is not split at $\mathfrak p$ and $a \not \in \mathcal{O}_{\mathfrak p}^\times$. Hence $a \not\in T((M_i, \sigma_i, a))^\times$ by Proposition \ref{propositionTIsSemilocal} in contradiction to our assumption on $a$.

  Now consider the case of a proper extension $L/K$ of degree $n$ with $l$ admissible. 
  If $L/K$ is inseparable and $l = \kar K$, then Proposition \ref{propInseparableExtensionSplitsAllAlgebras} implies that all algebras $(M_i, \sigma_i, a)$ are split over $L$, so any choice of $a$ will do as long as $(a) \in I_{\mathfrak m}$, $(a) \not\in H$. Such $a$ is afforded by Lemma \ref{lemmaSplitInExtension}.

  If $L/K$ is separable, let $P \subseteq \Sigma \setminus \Sigma_\infty$ be the set of $l$-good primes; it has density at least $\frac 1 {n!}$. Therefore Lemma \ref{lemmaSplitInExtension} is applicable, so we obtain $a \in K^\times$ such that $(a) \in I_{\mathfrak m}$, $(a) \not\in H$ and all places $\mathfrak p \in \Sigma \setminus \Sigma_\infty$ such that $a \not\in \mathcal{O}_{\mathfrak p}^\times$ are in $P$.
  We claim that $a$ is as desired, so we must show that \[ a, \frac{1}{a} \in T((M_i, \sigma_i, a) \otimes_K L / L) \] for all $i$. The algebras $(M_i, \sigma_i, a) \otimes_K L$ split at all infinite places of $L$ by construction of the $M_i$, so by Proposition \ref{propositionTIsSemilocal} it suffices to show that they split at all primes $\mathfrak q$ of $L$ above primes $\mathfrak p \in \Sigma \setminus \Sigma_\infty$ with $a \not\in \mathcal{O}_{\mathfrak p}^\times$.
  But all those $\mathfrak p$ are $l$-good, so $l \mid [L_{\mathfrak q} \colon K_{\mathfrak p}]$ and hence $L_{\mathfrak q}$ does split all $(M_i, \sigma_i, a)$ by the theory of central simple algebras over local fields.
\end{proof}

\begin{remark}\label{remarkOnlyIdealMatters}
  The element $a$ in the statement $(\dagger)^l_{L/K}$ can be multiplied by an arbitrary element of $\mathcal{O}_K^\times$, i.e.\ the statement is really one about the principal ideal $(a)$.
To see this, observe that $T$ is invariant under multiplication by $\mathcal{O}_K^\times$, and the local splitting behaviour of $(M_i, \sigma_i, a)$ at a prime $\mathfrak p$ unramified in $M/K$ only depends on the valuation $v_{\mathfrak p}(a)$, since the local norm group contains the local unit group for unramified extensions.
\end{remark}

For each class of ideals in the set $(I_{\mathfrak m}/P_{\mathfrak m}) \setminus (H/P_{\mathfrak m})$ that contains a principal (fractional) ideal, fix a representative principal ideal $(a_j)$ and a generator $a_j \in K^\times$ thereof. This is a finite list since $I_{\mathfrak m}/P_{\mathfrak m}$ is finite.
Thus every principal ideal in $I_{\mathfrak m} \setminus H$ has the form $(a_j b)$ for some $b \in U_{\mathfrak m}$ and one of the $a_j$.
Therefore, by Remark \ref{remarkOnlyIdealMatters}, we may rephrase the statement $(\dagger)^l_{L/K}$ as follows:
  \begin{quotation}
    For some $a_j$, there exists a $b \in U_{\mathfrak m}$ such that for all $i$ we have
    $a_j b, \frac{1}{a_j b} \in T((M_i, \sigma_i, a_j b) \otimes_K L / L)$.
  \end{quotation}

  This statement is of a very specific form; in fact, we will show that is equivalent to a certain system of polynomial equations $G_r(x_1, \dotsc, x_s, y_1, \dotsc, y_t) = 0$ having a solution in $K^s \times L^t$. We again adopt the viewpoint of first-order logic in phrasing and proving this expressibility result:

\begin{lemma}\label{lemmaCriterionIsFirstOrder}
  Consider the first-order language of pairs of rings, i.e.\ with signature $(+, \cdot, 0, 1, U)$, where $U$ is a unary predicate for a distinguished subring.
  There exists a positive existential sentence $\psi_{K,n,l}$ in this language, with parameters from $K$, such that the condition $(\dagger)^l_{L/K}$ from Proposition \ref{propositionCriterionForProperExtension} is expressed precisely by $(L,K) \models \psi_{K,n,l}$.
\end{lemma}
\begin{proof}
  We use the equivalent form of $(\dagger)^l_{L/K}$ introduced above.
  This statement is straightforwardly written as
  \[ \psi_{K,n,l} = \bigvee_j \exists b \Big( b \in U_{\mathfrak m} \wedge \bigwedge_i a_j b, \frac 1 {a_j b} \in T((M_i, \sigma_i, a_j b))  \Big),  \]
  where $b \in U_{\mathfrak m}$ can be phrased as a positive existential statement since $U_{\mathfrak m}$ is a diophantine subset of $K$ by Corollary \ref{corollaryOpDefinable}, and $a_j b \in T((M_i, \sigma_i, a_j b))$ can likewise be expressed by Lemma \ref{lemmaCyclicTDefinable}.
\end{proof}

\begin{theorem}\label{theoremExistentialFormulaForProperExtension}
  There exists a positive existential sentence $\psi_{K,n}$ in the language of pairs of rings, with parameters from $K$, such that $(K,K) \models \neg\psi_{K,n}$, but $(L,K) \models \psi_{K,n}$ for all extensions $L/K$ of degree $n$.
\end{theorem}
\begin{proof}
  Let $\psi_{K,n} = \bigvee_{l \mid n} \psi_{K,n,l}$. Now the statement is an immediate consequence of Proposition \ref{propositionCriterionForProperExtension}, Lemma \ref{lemmaCriterionIsFirstOrder} and Lemma \ref{lemmaGoodness}.
\end{proof}

\section{Proof of the main results}

\begin{lemma}\label{lemmaCriterionForNoRoot}
  Let $f \in K[X]$ be monic of degree $n > 1$, so $K[X]/(f)$ is a ring into which $K$ embeds canonically.
  Then $f$ has no root in $K$ if and only if either $(K[X]/(f), K) \models \psi_{K,n}$ or $f$ factors as $f = g \cdot h$ with $g,h \in K[X]$ of degree $< n - 1$ with both $g$ and $h$ not having a root in $K$.
\end{lemma}
\begin{proof}
  Assume that $f$ has a root $x \in K$. Then $K[X]/(X-x) \cong K$ is a homomorphic image of $K[X]/(f)$ preserving $K$. Since $(K,K) \models \neg\psi_{K,n}$ by Theorem \ref{theoremExistentialFormulaForProperExtension}, we obtain $(K[X]/(f),K) \not\models \psi_{K,n}$, since truth of positive existential sentences is preserved under taking homomorphic images. Furthermore $f$ cannot factor as a product of two polynomials in $K[X]$ without roots in $K$.

  For the converse direction, assume that $f$ has no root in $K$. Then either $f$ can be written as a product of two polynomials in $K[X]$ without roots (and therefore each of degree $> 1$), or $f$ is irreducible. In the latter case $K[X]/(f)$ is a field of degree $n$ over $K$, so $(K[X]/(f), K) \models \psi_{K,n}$ by Theorem \ref{theoremExistentialFormulaForProperExtension}.
\end{proof}

Now we can prove our main theorem.
\begin{theorem}\label{theoremMainThm}
  There exists an existential first-order formula $\varphi_{K,n}(a_0, \dotsc, a_{n-1})$ in the language of rings, with parameters in $K$, such that $K \models \varphi_{K,n}(a_0, \dotsc, a_{n-1})$ if and only if the polynomial $f = X^n + a_{n-1}X^{n-1} + \dotsb + a_0$ has no root in $K$.
\end{theorem}
\begin{proof}
  We translate the equivalent statement from Lemma \ref{lemmaCriterionForNoRoot}. Note that the structure $(K[X]/(f), K)$ is quantifier-freely definable in $K$, since elements of $K[X]/(f)$ correspond in a straightforward way to $n$-tuples of elements of $K$, and the definitions of addition, multiplication, and the distinguished subset $K$ are immediate.

  Hence we obtain $\varphi_{K,2}(a_0,a_1)$ by rewriting $\big(K[X]/(X^2 + a_1X + a_0), K\big) \models \psi_{K,2}$ as a statement about $K$ and similarly $\varphi_{K,3}(a_0,a_1,a_2)$, since polynomials of degree at most $3$ cannot factor as polynomials of smaller degrees without roots.
  For $\varphi_{K,4}$, we have to allow for the possibility of a reducible polynomial of degree $4$ without roots, so we rewrite the statement
  \begin{multline*} \Big((K[X] / (X^4 + a_3X^3 + X^2a_2 + Xa_1 + x_0), K) \models \psi_{K,4}\Big) \vee \\ \exists b_0,b_1, c_0,c_1 \Big( X^4 + a_3X^3 + a_2X^2 + a_1X + a_0 = (X^2 + b_1X + b_0)(X^2 + c_1X + c_0) \wedge \\ \varphi_{K,2}(b_0,b_1) \wedge \varphi_{K,2}(c_0,c_1) \Big) \end{multline*}
  as a first-order statement about $K$ to obtain $\varphi_{K,4}$; this is again correct by Lemma \ref{lemmaCriterionForNoRoot}.
  Inductively, we can construct $\varphi_{K,n}$ in this manner for all $n$.
\end{proof}

\begin{corollary}
    Let $K^{\ast\ast} \supseteq K^\ast$ be any two fields which are both elementary extensions of $K$. Then $K^\ast$ is relatively algebraically closed in $K^{\ast\ast}$.
\end{corollary}
\begin{proof}
  Theorem \ref{theoremMainThm} is also true in $K^\ast$ and $K^{\ast\ast}$ -- with the same formulae $\varphi_{K,n}$ -- by first-order transfer.
  Let $f = X^n + a_{n-1}X^{n-1} + \dotsb a_0 \in K^{\ast}[X]$ be a polynomial without a root in $K^\ast$. Then $K^{\ast} \models \varphi_{K,n}(\mathbf a)$, therefore $K^{\ast\ast} \models \varphi_{K,n}(\mathbf a)$ (since $\varphi_{K,n}$ is an existential formula), whence $f$ does not have a root in $K^{\ast\ast}$ either. 
\end{proof}

\begin{corollary}
  There exists a diophantine criterion for a polynomial over $K$ in an arbitrary number of variables to be irreducible.
\end{corollary}
\begin{proof}
  Irreducibility is expressible by a universal first-order formula, since $f$ being irreducible means that for all pairs of polynomials of smaller total degree, $f$ is not equal to their product. By the Łoś-Tarski Preservation Theorem of model theory (\cite[Corollary 5.4.5]{shorterModelTheory}), this property is expressible by an existential first-order formula with parameters if and only if for every $K^{\ast\ast} \supseteq K^\ast$ with $K^{\ast\ast}, K^\ast \succeq K$ every irreducible polynomial over $K^\ast$ remains irreducible over $K^{\ast\ast}$.

  This condition is a simple consequence of relative algebraic closedness: Consider an irreducible polynomial $f \in K^\ast[\mathbf X]$, and assume without loss of generality (after affine change of coordinates and rescaling) that $f$ has constant coefficient $1$. Then $f$ factors into irreducible factors $f_1, \dotsc, f_n \in \overline{K^\ast}[\mathbf X]$, each with constant coefficient $1$, and these factors remain irreducible in $\overline{K^{\ast\ast}}[\mathbf X]$. If $f$ factors non-trivially as $g \cdot h$ in $K^{\ast\ast}[\mathbf X]$, we may assume after rescaling that both $g$ and $h$ have constant coefficient $1$, so $g, h$ can be factored into products of the $f_i$ in $\overline{K^{\ast\ast}}[\mathbf X]$ since this is a unique factorisation domain. But then the coefficients of $g$ and $h$ are both in $\overline{K^\ast}$ and in $K^{\ast\ast}$, so they are in $K^\ast$, contradicting $f$ being irreducible.
\end{proof}

\bibliographystyle{amsalpha}
\bibliography{../Bibliography}

\end{document}